\documentclass{amsart}
\usepackage[left=2cm,right=2cm,top=2cm,bottom=2cm]{geometry}
\usepackage{amssymb,graphicx,url}

\newtheorem{theorem}{Theorem}[section]
\newtheorem{lemma}[theorem]{Lemma}
\newcommand{\C}{\mathbb{C}}
\newcommand{\Q}{\mathbb{Q}}
\newcommand{\R}{\mathbb{R}}
\newcommand{\X}{\mathbb{X}}
\newcommand{\Z}{\mathbb{Z}}
\newcommand{\h}{\mathbb{H}^2}
\newcommand{\F}{\mathcal{F}}
\newcommand{\supp}{\mathrm{supp}}
\newcommand{\gl}[1]{\mathrm{GL}(#1)}
\newcommand{\psl}[2]{\mathrm{PSL}_{#1}(#2)}
\DeclareMathOperator{\Id}{Id}
\DeclareMathOperator{\arccosh}{arccosh}
\DeclareMathOperator{\area}{Area}
\DeclareMathOperator{\ax}{ax}
\DeclareMathOperator{\cs}{cs}
\DeclareMathOperator{\spec}{spec_{\text{discr}}}
\newcommand{\eps}{\varepsilon}
\newcommand{\abs}[1]{\left\vert#1\right\vert}

\newcommand{\ignore}[1]{}

\begin{document}

\title{Cheeger constants of hyperbolic reflection groups and Maass cusp forms
of small eigenvalues}
\author[B.~A.\ Benson]{Brian~A.\ Benson}\address{Department of Mathematics,
University of California, Riverside, 900 University Avenue, Riverside,
CA 92521, email: bbenson@ucr.edu}
\author[G.~S.\ Lakeland]{Grant~S.\ Lakeland}\address{Department of Mathematics
\& Computer Science, Eastern Illinois University, 600 Lincoln Avenue,
Charleston, IL 61920, email: gslakeland@eiu.edu}
\author[H.~Then]{Holger Then}\address{Freie Waldorfschule Augsburg,
Dr.-Schmelzing-Stra\ss e~52, 86169~Augsburg, Germany,
e-mail: holger.then@bristol.ac.uk}

\begin{abstract}
We compute the Cheeger constants of a collection of hyperbolic surfaces
corresponding to maximal non-compact arithmetic Fuchsian groups, and to
subgroups which are the rotation subgroup of maximal reflection groups.
The Cheeger constants are geometric quantities, but relate to the
smallest eigenvalues of Maass cusp forms. From geometrical considerations,
we find evidence for the existence of small eigenvalues. We search for
these small eigenvalues and compute the corresponding Maass cusp forms
numerically.
\end{abstract}

\date{July 20, 2019}

\thanks{G.~S.~L.\ thanks Chris Leininger and Alan Reid for helpful
conversations.}

\maketitle

\section{Introduction}\label{s:intro}

Given a hyperbolic Riemann surface $M = \Gamma\backslash\h$ for a cofinite,
discrete Fuchsian group $\Gamma\subset\psl2\R$, the \emph{Cheeger constant} of
$M$ is defined as
\begin{align*}
h(M) = \inf_{E \subset M} \dfrac{\ell(E)}{\min\left\{ \area(A), \area(B)
\right\}}
\end{align*}
where $E$ varies over all $1$-dimensional subsets of $M$ which divide $M$ into
two disjoint components $A$ and $B$, and $\ell(E)$ denotes the length of $E$.
This quantity is related to the geometry of $M$ and was introduced by Cheeger
\cite{Cheeger}. The quantity is also related to the first eigenvalue
$\lambda_1$ of the Laplacian operator on $M$; when $M$ is a closed Riemannian
manifold (i.e., $M$ has no cusps or cone points, or equivalently $\Gamma$ has no
parabolic or elliptic elements), $h$ and $\lambda_1$ are related via the inequalities
\begin{align*}
\dfrac{h^2(M)}{4} \leq \lambda_1(M) \leq 2h(M) + 10h^2(M),
\end{align*}
where the first inequality is due to Cheeger \cite{Cheeger}, and the second to
Buser \cite{Buser}. When $M$ is finite area, but permitted to be non-compact and have cone points, 
one has that 
\begin{align*}
\dfrac{h^2(M)}{4} \leq \lambda(M) \leq Ch(M)
\end{align*}
where $\lambda(M)=\inf_{f}\tfrac{\int_M \|\nabla f\|^2 \, dA}{\int_M f^2 \, dA}$
with $\area(\supp(f))>0,$ see \cite[Section 7]{Buser}. In particular, if a sequence of surfaces $M_n$ have
$h(M_n)$ converging to $0$ as $n$ diverges to $\infty$, then also
$\lambda_1(M_n)$ goes to $0$, and vice-versa.

Due to its definition as an infimum, it is often possible to bound the Cheeger
constant from above by finding a separating subset $E$, but computing the
constant precisely is a more complicated endeavor. The author of \cite{Benson1}
recently provided an algorithm which enables practical computation of a
wide range of examples of hyperbolic surfaces. The examples presently of
interest are (the orientation-preserving subgroups of) maximal arithmetic
reflection groups; specifically, we study those which are not maximal
Fuchsian groups. In \cite{LakelandThesis} it was shown that eight of the ten
such maximal reflection groups have the property that they are not
\emph{congruence} (see section~\ref{s:preliminaries} for definition).

Belolipetsky \cite{Belo} asked about congruence in the context of obtaining
a practical lower bound on $\lambda_1$ for maximal arithmetic reflection
groups, because congruence groups have $\lambda_1$ bounded below by
$975/4096$ (Kim and Sarnak \cite{KimSarnak}), and conjecturally by $1/4$
(Selberg \cite{Selberg}). We therefore seek to understand the geometry of
those examples which are not congruence, with a view to establishing whether
their values of $\lambda_1$ can be much smaller than these bounds. We prove the
existence of small
eigenvalues by searching for them and we compute the corresponding
Maass cusp forms numerically. The smallest eigenvalue we find is approximately
$0.14843$, and in total we find that six of our non-congruence examples have $\lambda_1$ less than $1/4$.

We note that all of the examples considered here have spectral gap at most $1/4$, on account of the fact that the continuous spectrum of the Laplacian is $[ 1/4 , \infty)$ \cite{Roelcke1966} (see Section 4). We therefore compute the first discrete eigenvalue $\lambda_1^{\mbox{disc}}$. We find two examples where the first discrete eigenvalue $\lambda_1^{\mbox{disc}}(M)$ is much larger than $h(M).$
 We find one further example where $\lambda_1^{\mbox{disc}}(M) > h(M),$ but the values are still similar, and the 
 remaining seven examples have $\lambda_1^{\mbox{disc}}(M) < h(M)$. This suggests that the three examples (one congruence, two non-congruence) with large $\lambda_1^{\mbox{disc}}$ are distinguished geometrically from
the other seven.


Since the examples in question are all index two subgroups of maximal
arithmetic Fuchsian groups, which are necessarily all congruence, we also
compute the Cheeger constant of the maximal Fuchsian group, with a view to
studying how the Cheeger constant changes when passing to the subgroup. We find
that the Cheeger constant seems to drop more dramatically when passing from a
congruence group to a non-congruence subgroup than when passing to a congruence
subgroup. More precisely, we find that in seven of the eight cases where the
subgroup is non-congruence, $h$ drops by more than half, and in the two cases
where the subgroup is congruence, $h$ drops by less than half.

This paper is organized as follows. In section~\ref{s:preliminaries}, we
provide some background and prove some helpful technical results. The
algorithm that will be used to compute the Cheeger constant
is described in section~\ref{s:algorithm}, and it is implemented in detail in
sections~\ref{s:examples}--\ref{sec 3.4}. Section~\ref{s:results} summarizes
our results on the Cheeger constant. In section~\ref{sec 4}, we compute
and investigate Maass cusp forms.
\ignore{The algorithm for computing Maass cusp forms is given in
appendix~\ref{sec A}.}
Our conclusions are stated in section~\ref{sec 5}.

\section{Preliminaries}\label{s:preliminaries}

\subsection{Arithmetic reflection groups}\label{sec 2.1}

A cofinite hyperbolic reflection group is the group generated by reflections
in the sides of a finite area hyperbolic polygon. Such a group is called
\emph{maximal} if it is not properly contained in another reflection group.
The index two orientation-preserving subgroup of such a group will be
referred to as its \emph{rotation subgroup}; this is a Fuchsian group. A
non-compact, cofinite, hyperbolic reflection group is \emph{arithmetic} if
it (and its rotation subgroup) is commensurable with the modular group
$\psl2\Z$; that is, if the reflection group has a finite index subgroup which
is ($\psl2\R$-conjugate to) a finite index subgroup of $\psl2\Z$. Although there
are infinitely many arithmetic reflection groups, it is known that there are
only finitely many maximal arithmetic reflection groups: in dimension two
this is due to Long, Maclachlan and Reid \cite{LMR}; in dimension three to
Agol \cite{Agol}; and in other dimensions independently to Agol,
Belolipetsky, Storm, and Whyte \cite{ABSW}, and to Nikulin \cite{Nikulin}.
Note that this includes the result that arithmetic reflection groups do not
exist above dimension 30.

For any positive integer $n$, there is a natural homomorphism
\begin{align*}
\psi_n : \psl2\Z \to \psl2{\Z/n\Z}
\end{align*}
defined by reducing the matrix modulo $n$. The kernel of this homomorphism is
a finite index subgroup of $\psl2\Z$ consisting of matrices congruent to
$\pm\Id$; this subgroup is denoted $\Gamma(n)$ and is called the
\emph{principal congruence subgroup} of level $n$. A Fuchsian group
commensurable with $\psl2\Z$ is called a \emph{congruence} group if it
contains some principal congruence subgroup.

It was shown by Helling \cite{Helling1, Helling2} that the maximal Fuchsian
groups commensurable with the modular group are obtained by taking the
normalizers $N(\Gamma_0(n))$ in $\psl2\R$ of the groups
\begin{align*}
\Gamma_0(n) = \left\{ \begin{pmatrix}a & b \\ c & d \end{pmatrix}
\in \psl2\Z \mid c \equiv 0 \mod n \right\}
\end{align*}
for $n$ a square-free integer. The modular group itself, as well as the
twelve normalizers $N(\Gamma_0(n))$ for
$n = 2, 3, 5, 6, 7, 10, 13, 14, 21, 30, 34, 39$ are rotation subgroups
of reflection groups \cite{LakelandThesis}. In these cases, the reflection
groups are necessarily maximal arithmetic reflection groups, and all are
necessarily congruence groups. The ten normalizers $N(\Gamma_0(n))$
for $n= 11, 15, 17, 19, 22, 26, 33, 42, 55, 66$ are not rotation subgroups of
reflection groups, but each contains an index two subgroup which is. Again,
these reflection groups are necessarily maximal arithmetic, but they are not
necessarily congruence groups, and in fact, only the examples corresponding
to $n=15$ and $n=17$ are congruence groups. The other eight examples are not
congruence groups. These ten examples are the focus of the present study.

\subsection{Fuchsian groups and quotient surfaces}\label{sec 2,2}

In this section, we give some elementary results which will be useful in
computing the Cheeger constants of our examples. We refer to Beardon
\cite{Beardon} for background on the material in this section.

\begin{lemma}\label{domainlemma}Suppose that $\Gamma$ is a cofinite Fuchsian
group and $\F$ is a convex, finite-sided fundamental domain for the action
of $\Gamma$ on $\h$. Suppose that there exists a finite collection of
hyperbolic elements $\{ \gamma_i \}_{i=1}^k \subset \Gamma$ such that:
\begin{itemize}
\item for each $i$, $\gamma_i$ pairs two sides, $s_i$ and $s_i'$, of $\F$;
\item for each $i$, $\ax(\gamma_i)$ intersects both $s_i$ and $s_i'$
orthogonally; and
\item for each $i \neq j$, $\ax(\gamma_i) \cap \ax(\gamma_j) = \emptyset$.
\end{itemize}
Suppose further that the axes $\ax(\gamma_i)$ divide $\F$ into two regions,
dark and light, such that each dark side is paired with another dark side,
and each light side is paired with another light side. Then the images of
the axes $\ax(\gamma_i)$ are geodesics on the quotient surface
$M=\Gamma\backslash\h$ which together separate $M$.\end{lemma}

\begin{proof} Since the axes are geodesics in $\h$, they project to
geodesics on the quotient surface. Since $\ax(\gamma_i)$ intersects
both $s_i$ and $s_i'$ orthogonally, gluing $s_i$ to $s_i'$ closes up
the geodesic. The disjointness of the axes implies that the corresponding
geodesics are disjoint on the quotient surface. Since like-shaded sides
are identified with like, the only way to move from light side to dark is
to cross an axis. \end{proof}

\begin{lemma}\label{invlength}Let $\alpha_1$ and $\alpha_2$ be elliptic
isometries of order 2 (henceforth called involutions) fixing $z_1$ and
$z_2 \neq z_1$ respectively. Then the product $\alpha = \alpha_1 \alpha_2$
is hyperbolic, the axis $\ax{( \alpha )}$ is the geodesic through $z_1$
and $z_2$, and the translation length of $\alpha$ along its axis is twice
the distance between $z_1$ and $z_2$.\end{lemma}

This is found on page~174 of Beardon.

\begin{lemma}\label{noshortgeods}Let $\Gamma$ be a non-cocompact, cofinite
Fuchsian group with finite-sided fundamental domain $\F$, and let $\ell >0$.
Choose disjoint $\Gamma$-equivariant cusp horoball neighborhoods, and let
$\F'$ denote (the closure of) the complement of these neighborhoods in $\F$.
Let $N(\F')=N_{\frac{\ell}{2}+\varepsilon}(\F')$ denote the closed
$(\frac{\ell}{2}+\varepsilon)$-neighborhood of $\F'$, and let
$G= \{ \gamma_0=\Id, \gamma_1, \ldots, \gamma_k \}$ be a set of elements
of $\Gamma$ such that the union of translates $\cup_{i=0}^k \gamma_i(\F)$
covers $N(\F')$. Then any geodesic of length at most $\ell$ on
$\Gamma\backslash\h$ must correspond to an element $\gamma_j \gamma_i^{-1}$
for some $0 \leq i, j \leq k$. \end{lemma}

\begin{proof}Let $\beta$ be a closed geodesic on $\Gamma\backslash\h$ with
length $ d \leq \ell$. Consider the lifts of $\beta$ to the universal cover
$\h$. Since $\beta$ is not peripheral, it cannot lie entirely in a single
cusp neighborhood, and since $\beta$ is path-connected, it cannot lie
entirely in the union of the disjoint cusp neighborhoods. Thus each lift
of $\beta$ has a point outside of the union of the equivariant horoball
neighborhoods. There must be one lift which intersects $\F'$; let $\beta'$
be such a lift, and choose a point $p_0 \in \beta' \cap \F'$.

The geodesic $\beta'$ is the axis of a hyperbolic isometry
$\gamma \in \Gamma$ with translation length $d$. There are two points
on $\beta'$, $p_1$ and $p_2$, each distance $\frac{d}{2}$ away from
$p_0$, such that $\gamma(p_1) = p_2$. If either $p_1$ or $p_2$ is fixed
by a finite order element of $\Gamma$, make a different choice of $p_1$
which is still on the axis of $\gamma$ and within an
$\frac{\varepsilon}{2}$-neighborhood of the inital choice, such that
$p_1$ and $p_2 = \gamma(p_1)$ both have trivial stabilizers. Note that
$p_1, p_2 \in N_{\frac{\ell}{2}+\varepsilon}(\F')$, and $p_1$ and $p_2$ are
$\Gamma$-equivalent. As such, there exists a point $q \in \F$ and there
exist group elements $\gamma_{i_1}, \gamma_{i_2} \in G$ such that
$p_1 = \gamma_{i_1}(q)$ and $p_2 = \gamma_{i_2}(q)$. Since
$q=\gamma_{i_1}^{-1}(p_1)$, it follows that
$p_2 = \gamma_{i_2} \circ \gamma_{i_1}^{-1}(p_1)$. Since $p_1$ and
$p_2$ were chosen to have trivial stabilizers, the only element of
$\Gamma$ which sends $p_1$ to $p_2$ is $\gamma$, and so
$\gamma = \gamma_{i_2} \circ \gamma_{i_1}^{-1}$. Hence the original
geodesic corresponds to an element $\gamma_j \gamma_i^{-1}$, for some
$0 \leq i, j \leq k$, as required. \end{proof}

\begin{lemma}\label{coarea} A finite coarea, genus zero Fuchsian group $\Gamma$
with $t$ $\Gamma$-inequivalent cusps, and $r$ cone points of orders
$m_1, \ldots, m_r$ respectively has coarea
\[ 2\pi \left( -2 + t + \sum_{i=1}^r \left( 1 - \frac{1}{m_i} \right) \right).\]
\end{lemma}

This is stated as Theorem~10.4.3 of Beardon, with our $t$ corresponding to $s$
there, and the quantity referred to as $t$ there always being $0$ in our
examples. In this paper, the genus $g$ will always equal 0, and so the term
$2(g-1)$ in the Theorem~10.4.3 formula becomes simply $-2$ for us.

\section{The algorithm}\label{s:algorithm}

In this section, we state the algorithm which is used to compute the Cheeger
constant of our examples. This is an adapted version of the algorithm given
in \cite{Benson1}; in particular, it has been simplified to reflect the fact
that all the present examples are non-compact. We will keep track of two
quantities: $H$ will denote the current best estimate of the Cheeger
constant $h$, and begins by taking the value $H=1$ because $M$ is assumed
to have cusps and the isoperimetric ratio of a cusp neighborhood is $1$;
$U$ represents the current upper bound on the total length of geodesics
(or curves equidistant from a collection of geodesics) that could possibly
result in a splitting which reduces $H$, and begins with the value
$U = \area(M)/2$.

\begin{enumerate}
\item First, set $H=1$ and $U= \area(M)/2$.
\item\label{step2} Select a collection $\{ \gamma^{i_1}, \ldots, \gamma^{i_j} \}$
of geodesics which split $M$ into two pieces $A$ and $B$, and which have
total length $\ell(\partial A) = \ell(\partial B)$ no greater than $U$.
\item If $\area(A) = \area(B)$, then compute
\begin{align*}
H_0 = h^*(A) = \dfrac{\ell(\partial A)}{\area(A)} = h^*(B) =
\dfrac{\ell(\partial B)}{\area(B)},
\end{align*}
let $s=0$, and proceed to Step~\ref{step5}.
\item If $\area(A) \neq \area(B)$, without loss of generality let $A$ denote
the component with lesser area. Determine the minimum distance $d_{i_j}$
perpendicular from the geodesics into $B$ before the neighborhoods
intersect, and minimize the maximum of
\begin{align*}
h^*(A_s) = \dfrac{\ell(\partial A) \cosh{(s)}}{\area(A)+\ell(\partial A)
\sinh{(s)}} \mbox{ and } h^*(B_s) =
\dfrac{\ell(\partial B) \cosh{(s)}}{\area(B)-\ell(\partial B)\sinh{(s)}}.
\end{align*}
Let this minimum be $H_0$ and record the value of $s$ for which this minimum
occurs.
\item\label{step5} If $H_0 < H$, then redefine $H=H_0$ and record the
collection $( \{ \gamma^{i_1}, \ldots, \gamma^{i_j} \}, s)$. If $H=H_0$,
add the collection to the list of collections which achieve $H$.
If $H_0 > H$, do nothing.
\item If $H \area(M)/2 < U$, redefine $U$ as $H \area(M)/2$; if not, leave
$U$ unchanged.
\item Return to Step~\ref{step2} until no further collections of geodesics
satisfying the criterion in Step~\ref{step2} can be found.
\end{enumerate}

In practice, we will present a splitting by geodesics, and then argue why
any other collection of geodesics which has length less than $U$ does not
result in a lesser value of $H$. As part of this, we will also check that
no union of geodesic arcs with endpoints at cone points produces a lesser
value of $H$.

To illustrate the general method, we show in detail the computation of the
Cheeger constants of the surfaces $M_{11}$ and $N_{11}$. We also show the
computation of the Cheeger constant of $M_{17}$ as one example where the
Cheeger constant is realized by a curve equidistant from a geodesic, rather
than by the geodesic itself.

\subsection{Example: $M_{11}$}\label{s:examples}

$M_{11}$ is the quotient of $\h$ by the normalizer $N(\Gamma_0(11))$ in
$\psl2\R$ of the congruence subgroup $\Gamma_0(11)$. A Ford domain $\F$
for $N(\Gamma_0(11))$ is given below in Figure~\ref{fig:domain01}.

\begin{figure}
\includegraphics[scale=0.75]{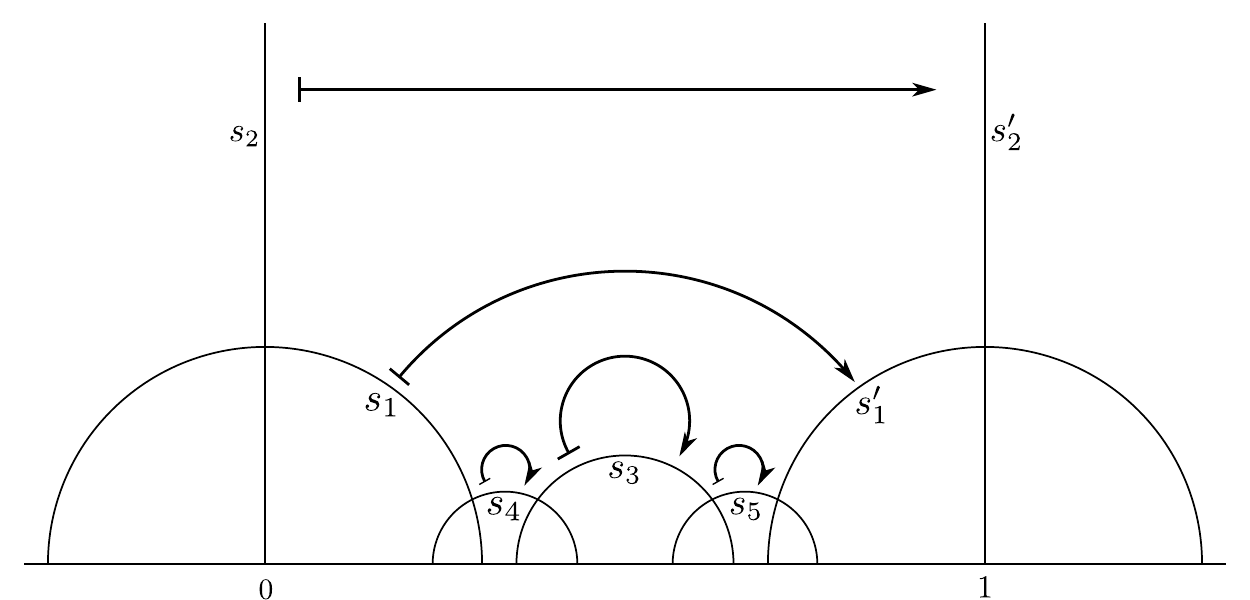}
\caption{\label{fig:domain01}Ford domain $\F$ for $N(\Gamma_0(11))$.}
\end{figure}

The domain $\F$ has area $2\pi$, and the side-pairing elements which
generate $N(\Gamma_0(11))$ are:
\begin{align*}
g_1 = \begin{pmatrix}\sqrt{11} & \frac{-1}{\sqrt{11}} \\
\sqrt{11} & 0 \end{pmatrix},\ g_2 =
\begin{pmatrix}1 & 1 \\ 0 & 1 \end{pmatrix},\ g_3
= \begin{pmatrix}\sqrt{11} & \frac{-6}{\sqrt{11}} \\
2\sqrt{11} & -\sqrt{11} \end{pmatrix},\ g_4 =
\begin{pmatrix}\sqrt{11} & \frac{-4}{\sqrt{11}} \\
3\sqrt{11} & -\sqrt{11} \end{pmatrix},\ g_5 =
\begin{pmatrix}2\sqrt{11} & \frac{-15}{\sqrt{11}} \\
3\sqrt{11} & -2\sqrt{11} \end{pmatrix}.
\end{align*}

The elements of $N(\Gamma_0(11))$ take two forms: those that belong to
$\psl2\Z$ have the form $\begin{pmatrix}a & b \\ 11c & d\end{pmatrix}$,
where $a, b, c, d \in \Z$; those that are not have the form
$\begin{pmatrix}x\sqrt{11} & \frac{y}{\sqrt{11}} \\
z\sqrt{11} & w\sqrt{11} \end{pmatrix}$,
where $x,y,z,w \in \Z$. As such, the traces of elements in $N(\Gamma_0(11))$
are either integers, or integer multiples of $\sqrt{11}$.

\begin{figure}
\includegraphics[scale=0.75]{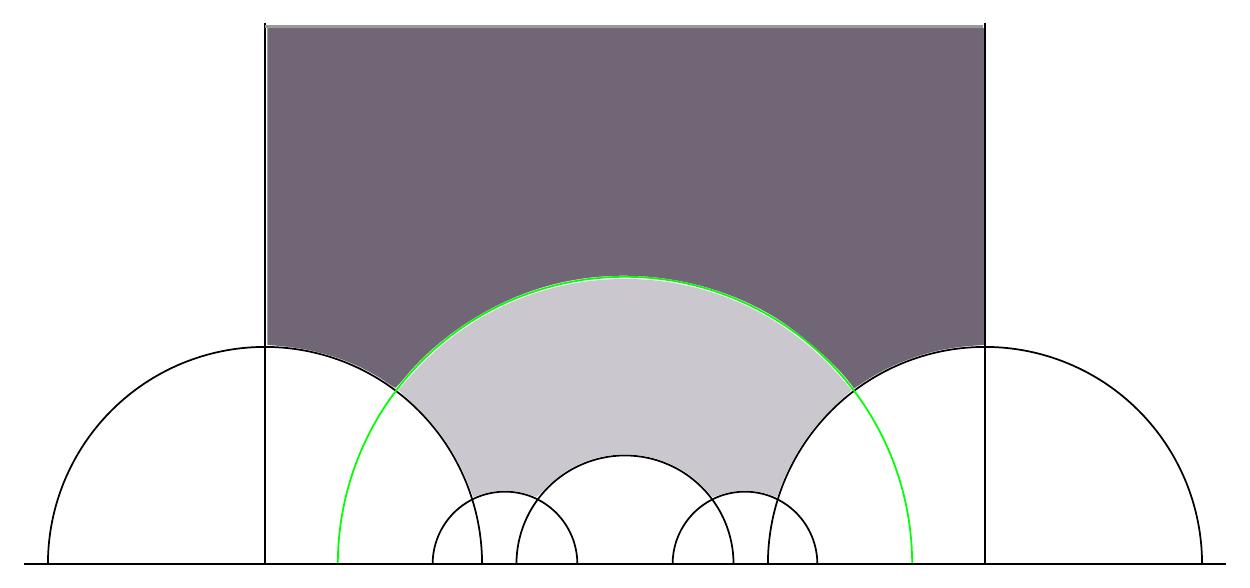}
\caption{\label{fig:domain02}The axis of $g_1$ separates $\F$.}
\end{figure}

To apply the algorithm, since $M_{11}$ has cusps, we set $U = \pi$ and $H=1$.
We consider the geodesic corresponding to the hyperbolic element $g_1$. This
element pairs the two sides $s_1$ and $s_1'$. Furthermore, the axis
$\ax(g_1)$ meets both $s_1$ and $s_1'$ orthogonally; this axis divides $\F$
into two pieces: the region $A$ above the axis (dark) and $B$ below (light);
each dark side is paired with another, and each light side is paired with
another (see Figure~\ref{fig:domain02}). Each of these two pieces has
area $\pi$, and the length of the separating geodesic corresponding to
$g_1$ is $2\arccosh{\left( \frac{\sqrt{11}}{2} \right)} \approx 2.18464$.
Hence our $H_0 = h^*(A) = h^*(B) = \frac{2.18464}{\pi} \approx 0.695394$.
This is less than $H=1$, so this value becomes our new $H$, and we set
$U=H\area(M)/2 \approx 2.18464$.

Since in this case, the two pieces have equal area, the Cheeger constant can
only be decreased if there exist simple closed curves, or a union of
geodesic arcs between cone points, which separate $M_{11}$ and which
have total length less than $U$, and hence the corresponding elements
have trace less than $\sqrt{11}$. We apply Lemma~\ref{noshortgeods} with
$\ell \approx 2.18464$, and we take $ \frac{\ell}{2}+\varepsilon = 1.1$.
We find that the set $G$ of elements whose translates cover the
$1.1$-neighborhood of $\F$ is
\begin{multline*}
G= \{ \Id, g_1, g_2, g_3, g_4, g_5, g_1^{-1}, g_2^{-1}, g_2^{-1}g_1,
g_2g_1^{-1}, g_1^{-1}g_5, g_4g_1^{-1}, g_5g_1, g_5g_3, g_1^{-1}g_3, \\
g_1^{-1}g_3g_4, g_1g_3, g_1g_3g_5, g_3g_1, g_3g_1g_4, g_3g_1^{-1},
g_3g_1^{-1}g_5 \}.
\end{multline*}
We search through elements of the form $\gamma_j \gamma_i^{-1}$ for
$\gamma_i, \gamma_j \in G$ for hyperbolic elements with trace less than
$\sqrt{11}$. We find that the only such elements are
\begin{align*}
\begin{pmatrix}5 & -1 \\ 11 & -2\end{pmatrix},\ \begin{pmatrix}9 & -5 \\
11 & -6\end{pmatrix},
\end{align*}
and their inverses. Each of these is a product of two involutions:
respectively, $g_3g_5g_3g_4$ and $g_5g_3g_4g_3$; and these elements are
conjugate to one another by $g_3$. As such, there is exactly one geodesic
arc on $M_{11}$, between two cone points, which corresponds to an element of
trace 3, and hence, by Lemma~\ref{invlength}, this arc has length
$\arccosh{\left( \frac{3}{2} \right)} \approx 0.962424$. The geodesic arc does
not intersect itself and its endpoints are distinct cone points, so the arc
does not separate $M_{11}$.

It remains to check whether there is another geodesic arc, between the same
cone points, which together with this one separates $M_{11}$. Specifically,
we seek geodesic arcs which are less than $2.18464-0.962424=1.22222$ in
length. However, by considering $1.25$-neighborhoods of each cone point,
we see that there is no other such geodesic arc. Furthermore, these
neighborhoods show that there are no two lifts of the same cone point
which are distance less than $U$ apart, and so there is no geodesic arc,
with both endpoints at the same cone points, of length less than $U$.
Hence we conclude that $h(M_{11}) \approx 0.695394$.

\subsection{Example: $N_{11}$}\label{sec 3.2}

The surface $N_{11}$ is the quotient of $\h$ by the index two subgroup
$\Gamma'\subset\Gamma$ generated by the following elements:

\begin{align*}
\gamma_1 = \begin{pmatrix}1 & 1 \\ 0 & 1 \end{pmatrix},\ \gamma_2 =
\begin{pmatrix}0 & \frac{-1}{\sqrt{11}} \\
\sqrt{11} & 0 \end{pmatrix},\ \gamma_3 =
\begin{pmatrix}\sqrt{11} & \frac{5}{\sqrt{11}} \\
2\sqrt{11} & \sqrt{11} \end{pmatrix},\ \gamma_4 =
\begin{pmatrix}10 & 3 \\ 33 & 10 \end{pmatrix},\ \gamma_5 =
\begin{pmatrix}23 & 8 \\ 66 & 23 \end{pmatrix}.
\end{align*}

A Ford domain, $Q$, for this group is given in Figure~\ref{fig:domain03};
for each $i$, $\gamma_i$ identifies side $s_i$ with $s_i'$. Note that these
sides $s_i$ are labeled differently than in the previous example.

\begin{figure}
\includegraphics[scale=0.75]{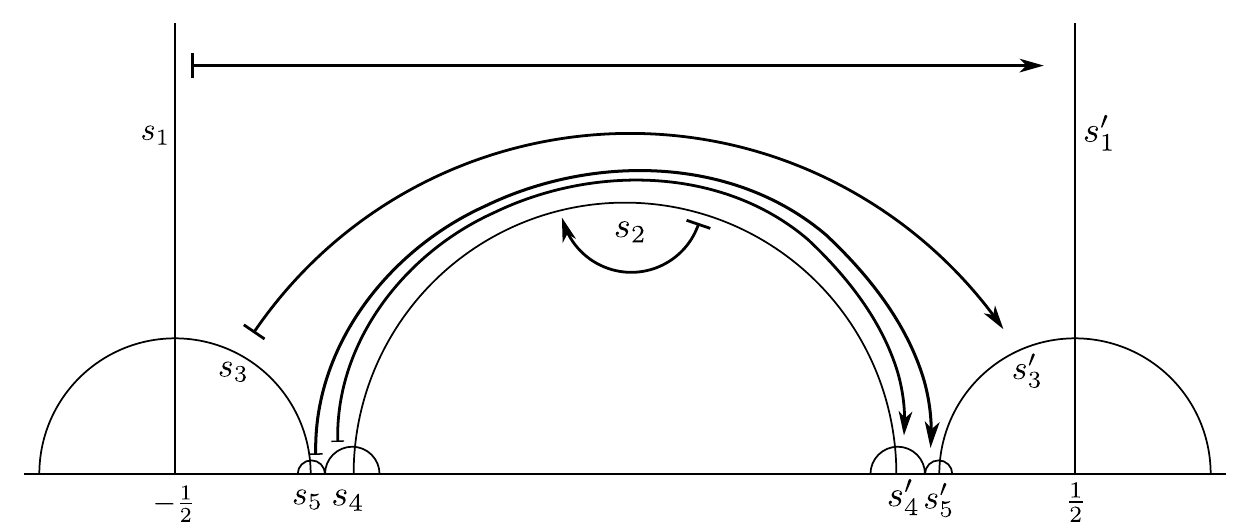}
\caption{\label{fig:domain03}Ford domain for $\Gamma'$.}
\end{figure}

Since $N_{11}$ has cusps, we set $U = 2\pi$ and $H=1$. We first choose to
study an element of $\Gamma'$ of trace 3, in particular the element
$\gamma_3 \gamma_2$. Since this element does not identify sides of $Q$,
we choose a different fundamental domain which has this element as a
side-pairing. We do this by taking the left half of $Q$, i.e., the part
of $Q$ with negative real part (henceforth called $Q^-$), and we apply
$\gamma_2$ to just $Q^-$. The resulting fundamental domain is shown in
Figure~\ref{fig:domain04}.

\begin{figure}
\includegraphics[scale=0.75]{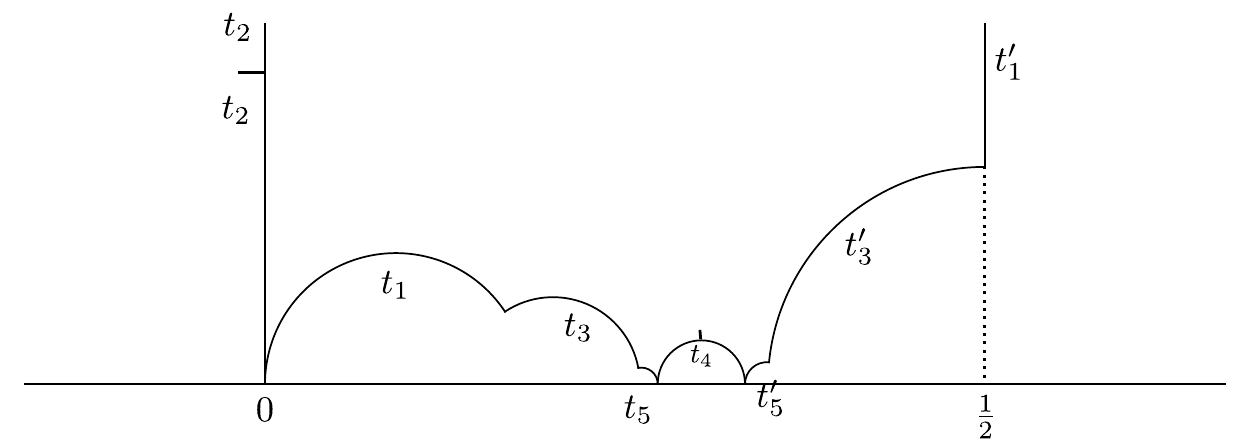}
\caption{\label{fig:domain04}Alternative fundamental domain for $\Gamma'$.}
\end{figure}

In this fundamental domain, the side $t_2$ is paired with itself by $\gamma_2$;
the sides $t_1$, $t_3$ and $t_5$ are identified with $t_1'$, $t_3'$ and $t_5'$
respectively, by $\gamma_1 \gamma_2$, $\gamma_3 \gamma_2$ and
$\gamma_5 \gamma_2$ respectively, and the side $t_4$ is paired with itself by
$\gamma_4 \gamma_2$. The axis of $\gamma_3 \gamma_2$ separates this fundamental
domain into dark and light pieces (see Figure~\ref{fig:domain05}).

\begin{figure}
\includegraphics[scale=0.75]{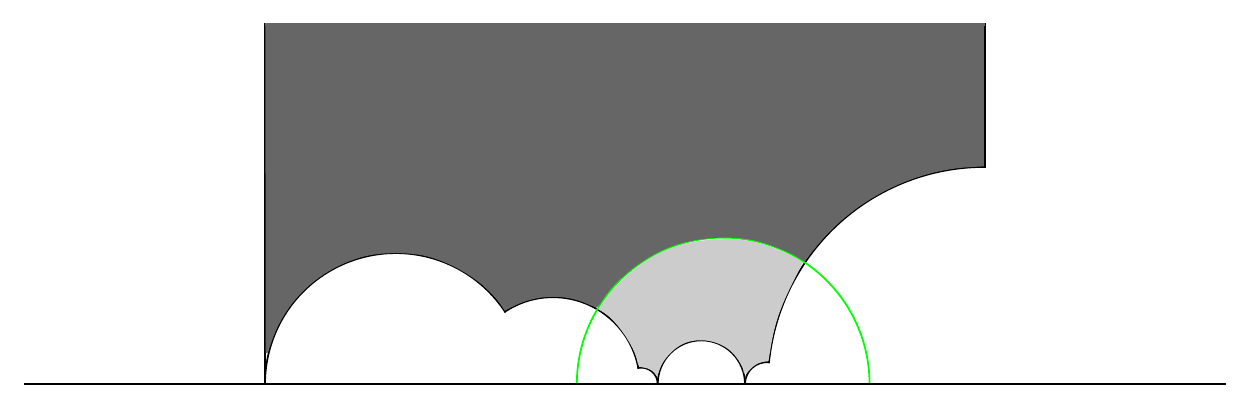}
\caption{\label{fig:domain05}The axis of $\gamma_3 \gamma_2$ separates
$N_{11}$.}
\end{figure}

Since $\gamma_3 \gamma_2$ has trace $3$, and its axis separates, it
corresponds to a geodesic of length
$2 \arccosh{\left( \frac{3}{2} \right)} \approx 1.92485$.
This geodesic separates $N_{11}$ into two pieces which each
have area $2\pi$. Our value of $H_0$ is therefore
$\frac{1.92485}{2\pi} \approx 0.306349$, and is equal to $h^*(A)$
and $h^*(B)$. This value is less than one, and hence becomes our
new $H$; we set $U = H\area(M)/2 \approx 1.92485$.

In this case, because the element has trace 3, there cannot be a shorter
geodesic simple closed curve on the surface, and because the two pieces
each have area $2\pi$, the minimum area is as large as possible. Therefore,
the only way in which this value could not be the true value of $h(N_{11})$
would be if there were arcs joining cone points, of total length less than
1.92485, which separate $N_{11}$. But by Lemma~\ref{invlength}, each pair
of cone points are at least distance $\frac{1.92485}{2}$ apart. We must
also account for the possibility that there is a geodesic arc with both
endpoints at the same cone points. However, since there was no such arc of
length less than $2.18$ on $M_{11}$, there cannot be one on $N_{11}$ of
length less than the present $U$. Hence we conclude that
$h(N_{11}) \approx 0.306349$.

\subsection{Other examples: the case $n=17$}\label{sec 3.3}

To illustrate other facets of the algorithm, we now detail the computation
of the Cheeger constant of $M_{17} = N(\Gamma_0(17))\backslash\h$ and of
$N_{17} = \Gamma'\backslash\h$ for the index 2 subgroup $\Gamma'$ of
$N(\Gamma_0(17))$ which is the rotation subgroup of a maximal arithmetic
reflection group.

Since $17$ is prime, the elements of $N(\Gamma_0(17))$ take two forms:
those that belong to $\psl2\Z$ have the form
$\begin{pmatrix}a & b \\ 17c & d\end{pmatrix}$, where $a, b, c, d \in \Z$;
those that are not have the form
$\begin{pmatrix}x\sqrt{17} & \frac{y}{\sqrt{17}} \\
z\sqrt{17} & w\sqrt{17} \end{pmatrix}$,
where $x,y,z,w \in \Z$. As such, the traces of elements in $N(\Gamma_0(17))$
are either integers, or integer multiples of $\sqrt{17}$. By considering the
congruence conditions on $a$ and $d$ from the restraint $ad-17bc=1$, and so
$ad \equiv 1$ mod $17$, we see that the least non-parabolic trace in
$\Gamma_0(17)$ is 5. Hence the shortest geodesic has trace $\sqrt{17}$,
and one such element is
$\begin{pmatrix} \sqrt{17} & \frac{-1}{\sqrt{17}} \\
\sqrt{17} & 0 \end{pmatrix}$.
This divides the surface into pieces $A$ and $B$ of areas $\pi$ and $2\pi$
(see Figure~\ref{fig:equid17}), and since these areas are not equal, we
proceed to extend $A$ into $B$ using closed curves equidistant from the
geodesic.

\begin{figure}
\includegraphics[scale=0.85]{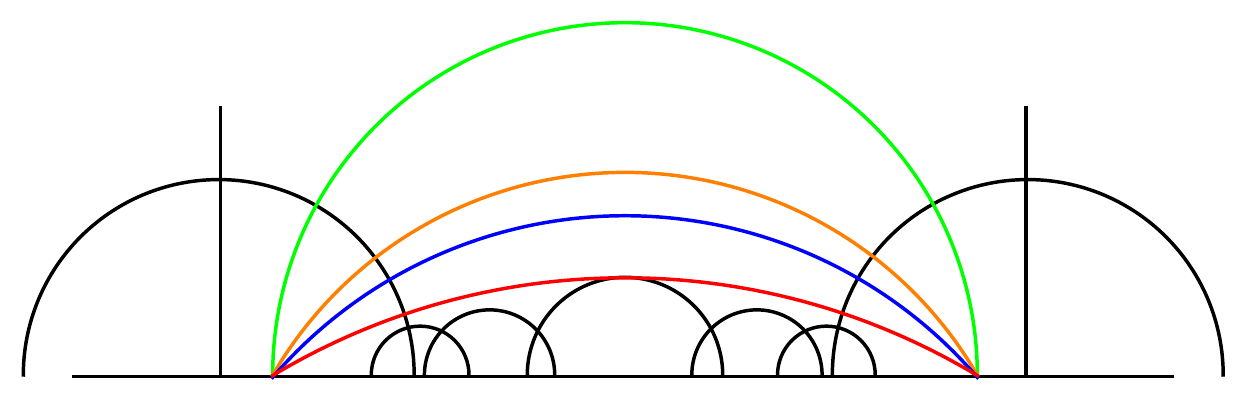}
\caption{\label{fig:equid17}Equidistant curves from the geodesic.}
\end{figure}

We let $s$ denote the distance extended into $B$. With reference to the
equidistant curves in the Ford domain in Figure~\ref{fig:equid17}, we find
that we may extend a distance of 1.28247 before the equidistant curve fails
to be isotopic to the geodesic. Using the equations
\begin{align*}
h^*(A_s) = \frac{\ell(\partial A) \cosh{(s)}}{\pi + \ell(\partial A)
\sinh{(s)}} \ \mbox{and} \ h^*(B_s) = \frac{\ell(\partial A)
\cosh{(s)}}{2\pi - \ell(\partial A) \sinh{(s)}}
\end{align*}
we find that $h^*(A_s)$ is minimized at $s= 0.779274$, and that
$h^*(A_s) = h^*(B_s)$ when $s = 0.55247$. Thus \\ $\max(h^*(A_s),h^*(B_s))$
is minimized at $s=0.55247$, and here $H_0 = h^*(A_s) = h^*(B_s) = 0.663522$.
Since this last number is less than $H=1$, we set $H=0.663522$ and
$U = 3.12677$.

We next note that any other geodesic of trace $\sqrt{17}$ would produce a
similar analysis. This is because the area formula (see Lemma~\ref{coarea})
means that such a
geodesic, if it separates, must divide the surface into subsurfaces of
areas $\pi$ and $2\pi$ respectively. Furthermore, as discussed above, the
next smallest possible trace for a hyperbolic element in this group is $5$,
which corresponds to a geodesic of length $3.1336$. This is larger than the
present value of $U$, and so we need not consider any such elements.

It therefore remains to consider geodesic arcs between cone points. By
considering neighborhoods of the fixed points of each involution and
their translates by the group, we see that the minimum trace for a product
of two involutions is $5$. By Lemma~\ref{invlength}, this means that the
union of two geodesic arcs between distinct cone points must have length
which exceeds $U$. It then remains to consider geodesic arcs which begin
and end at the same cone point. We find one product of conjugate involutions
with trace $19$; by Lemma~\ref{invlength}, this gives a geodesic arc of
length $2.94166$. This arc cuts off a subsurface of area less than $\pi$,
and so the corresponding value of $h^*$ is at least
$2.94166/\pi \approx 0.936359$. All other geodesic arcs have length exceeding
$U$. We therefore conclude that $h(M_{17}) \approx 0.663522$.

The surface $N_{17}$ is the quotient of $\h$ by an index two subgroup of
$N(\Gamma_0(17))$. This surface has area $6\pi$, so we begin by setting
$H=1$ and $U=3\pi$. We find a geodesic of trace $6$ which divides it into
two equal pieces, each of area $3\pi$. Here,
$H_0 = h^*(A) = h^*(B) \approx 0.374067$; this value becomes our new $H$,
and we set $U \approx 3.52549$. We find no shorter geodesics than this,
and we also find no combinations of arcs between cone points of total
length less than $U$. Thus $h(N_{17}) \approx 0.374067$.

\subsection{Other examples: the case $n=33$}\label{sec 3.4}

Since it exhibits different characteristics from the other values of $n$,
we will describe the results of the algorithm for the case of $n=33$.
Before doing this, we first describe how it differs from the other examples.

In all the cases we consider, the Cheeger constant of the minimal surface
$M_n$, corresponding to the maximal Fuchsian group $N(\Gamma_0(n))$, is
realized by either a closed geodesic (as in the case $n=11$) or an
equidistant neighborhood (as in the case $n=17$). In the other nine
cases where $n \neq 33$, the Cheeger constant of $N_n$ is realized as
a geodesic which passes through two of the cone points of $M_n$ which
do not lift (i.e., the involutions are not found in the index 2 subgroup).
In the case $n=33$, these cone points are a little farther apart, and so
the Cheeger constant of $N_{33}$ is realized by a different closed geodesic.

In the case of $M_{33}$, we find that the Cheeger constant is
$h(M_{33}) \approx 0.740622$ and this is realized by an equidistant curve
to a geodesic of trace $\sqrt{33}$. We omit the details of this
computation as they are similar to those of $M_{17}$.

To find $h(N_{33})$, we begin by taking the geodesic described above. It
has trace 14, and divides the surface into two pieces, each of area $4\pi$.
The corresponding Cheeger estimate is $H \approx 0.419201$, and
$U \approx 5.26783$. We now apply Lemma~\ref{noshortgeods} with
$\frac{\ell}{2}+\varepsilon = 2.64$, and search for elements with trace
less than 14. Up to conjugation, we find six geodesics, with traces
$\sqrt{33}$, $4\sqrt{3}$, $2\sqrt{33}$, $7\sqrt{3}$, $13$, and $8\sqrt{3}$
respectively. In all cases but trace 13, the geodesic does not divide the
surface into pieces of equal area, and hence the corresponding $H_0$ is
larger than $H$. In the case of the geodesic of trace 13, the geodesic
divides the surface into two pieces of equal area $4\pi$, and so the
corresponding $H_0 \approx 0.407274$ (see Figure~\ref{fig:n33}). Since
we already enumerated all of the closed geodesics of trace less than 14,
it remains to check that there are no combinations of arcs between cone
points with total length less than $U \approx 5.11796$. There are no such
arcs, so we conclude that $h(N_{33}) \approx 0.407274$.

\begin{figure}
\includegraphics[scale=0.995]{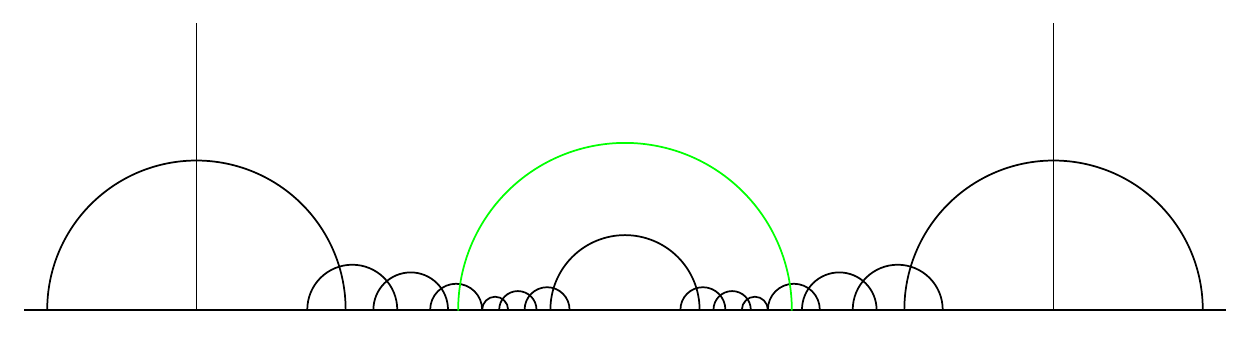}
\caption{\label{fig:n33}A trace 13 geodesic splits $N_{33}$ into two pieces of
area $4\pi$.}
\end{figure}

\subsection{Summary of results}\label{s:results}

Table~\ref{tab:Cheeger constants} lists numerical values of the Cheeger
constants of a collection of hyperbolic surfaces $M_n$ and $N_n$. These
values of the Cheeger constants result from the following considerations:

\begin{table}
\caption{\label{tab:Cheeger constants}Cheeger constants. The asterisk
$\mbox{}^\ast$ denotes cases where $N_n$ is a congruence surface.}
\begin{tabular}{l|cc}
$\,n$ & $h(M_n)$ & $h(N_n)$ \\
\hline
11 & 0.695394 & 0.306349 \\
$15^\ast\!$ & 0.814548 & 0.498728 \\
$17^\ast\!$ & 0.663522 & 0.374067 \\
19 & 0.672365 & 0.183809 \\
22 & 0.717333 & 0.279467 \\
26 & 0.719348 & 0.239543 \\
33 & 0.740622 & 0.407274 \\
42 & 0.772596 & 0.328406 \\
55 & 0.485381 & 0.204233 \\
66 & 0.472607 & 0.218937 \\
\end{tabular}
\end{table}

\begin{itemize}
\item \underline{$n=11$:} $M_{11}$ has area $2\pi$ and a geodesic with trace
$\sqrt{11}$ which divides it into two equal pieces; thus
$h(M_{11}) \approx 0.695394$. $N_{11}$ has area $4\pi$ and a geodesic
with trace $3$ which divides it into two equal pieces; thus
$h(N_{11}) \approx 0.306349$.
\item \underline{$n=15$:} $M_{15}$ has area $2\pi$ and a geodesic with trace
$\sqrt{15}$ which divides it into two equal pieces; thus
$h(M_{15}) \approx 0.814548$. $N_{15}$ has area $4\pi$ and a geodesic with
trace $5$ which divides it into two equal pieces; thus
$h(N_{15}) \approx 0.498728$.
\item \underline{$n=17$:} $M_{17}$ has area $3\pi$ and a geodesic with
trace $\sqrt{17}$ which divides it into two pieces of areas $\pi$ and
$2\pi$; an equidistant curve from this at distance $s=0.55247$ gives
$h(M_{17}) \approx 0.663522$. $N_{17}$ has area $6\pi$ and a geodesic
with trace $6$ which divides it into two equal pieces; thus
$h(N_{17}) \approx 0.374067$.
\item \underline{$n=19$:} $M_{19}$ has area $10\pi/3$ and a geodesic with
trace $\sqrt{19}$ which divides it into two pieces of areas $\pi$ and
$7\pi/3$; an equidistant curve from this at distance $s=0.685239$ gives
$h(M_{19}) \approx 0.672365$. $N_{19}$ has area $20\pi/3$ and a geodesic
with trace $3$ which divides it into equal pieces; thus
$h(N_{19}) \approx 0.183809$.
\item \underline{$n=22$:} $M_{22}$ has area $3\pi$ and a geodesic with
trace $\sqrt{22}$ dividing it into two pieces of areas $\pi$ and $2\pi$;
an equidistant curve from this at distance $s=0.503269$ gives
$h(M_{22}) \approx 0.717333$. $N_{22}$ has area $6\pi$ and a geodesic with
trace $4$ which divides it into equal pieces; thus
$h(N_{22}) \approx 0.279467$.
\item \underline{$n=26$:} $M_{26}$ has area $7\pi/2$ and a geodesic with
trace $\sqrt{26}$ dividing it into two pieces of areas $\pi$ and $5\pi/2$;
an equidistant curve from this at distance $s=0.686576$ gives
$h(M_{26}) \approx 0.719348$. $N_{26}$ has area $7\pi$ and a geodesic
with trace $4$ which divides it into equal pieces; thus
$h(N_{26}) \approx 0.239543$.
\item \underline{$n=33$:} $M_{33}$ has area $4\pi$ and a geodesic with trace
$\sqrt{33}$ dividing it into two pieces of areas $\pi$ and $3\pi$; an
equidistant curve from this at distance $s=0.820071$ gives
$h(M_{33}) \approx 0.740622$. $N_{33}$ has area $8\pi$ and a geodesic with
trace $13$ which divides it into equal pieces; thus
$h(N_{33}) \approx 0.407274$.
\item \underline{$n=42$:} $M_{42}$ has area $4\pi$ and a geodesic with trace
$\sqrt{42}$ dividing it into two pieces of areas $\pi$ and $3\pi$; an
equidistant curve from this at distance $s=0.772596$ gives
$h(M_{42}) \approx 0.771086$. $N_{42}$ has area $8\pi$ and a geodesic with
trace $8$ which divides it into equal pieces; thus
$h(N_{42}) \approx 0.328406$.
\item \underline{$n=55$:} $M_{55}$ has area $6\pi$ and a geodesic with trace
$3\sqrt{11}$ dividing it into two equal pieces of areas $3\pi$ and $3\pi$;
this gives $h(M_{55}) \approx 0.485381$. $N_{55}$ has area $12\pi$ and a
geodesic with trace $7$ which divides it into equal pieces; thus
$h(N_{55}) \approx 0.204233$.
\item \underline{$n=66$:} $M_{66}$ has area $6\pi$ and a geodesic with trace
$2\sqrt{22}$ dividing it into two equal pieces of areas $3\pi$ and $3\pi$;
this gives $h(M_{66}) \approx 0.472607$. $N_{66}$ has area $12\pi$ and a
geodesic with trace $8$ which divides it into equal pieces; thus
$h(N_{66}) \approx 0.218937$.
\end{itemize}

We note that the upper bounds for the Cheeger constants of $N_{22}$ and
$N_{26}$ which appear in \cite[Table~4.2]{LakelandThesis} are erroneous, as
they are underestimates of the true Cheeger constants.

\section{Maass cusp forms}\label{sec 4}

In this section, we compute Maass cusp forms on a collection of hyperbolic
surfaces $M_n$ and $N_n$ for various values of $n$.

Let $M=\Gamma\backslash\h$ be the quotient surface of a cofinite, but
non-cocompact Fuchsian group $\Gamma\subset\psl2\R$, and let
$\chi:\Gamma\to\C\setminus\{0\}$ be a multiplicative character.

\bigskip

A \emph{Maass form} on $(\Gamma,\chi)$ is a
\begin{enumerate}
\item[] real analytic, $f\in C^\infty(\h)$,
\item[] square-integrable, $f\in L^2(M)$,
\item[] automorphic, $\chi(\gamma)f(\gamma z)=f(z)\ \forall\gamma\in\Gamma$,
\item[] eigenfunction of the Laplace-Beltrami operator,
$-\Delta f(z)=\lambda f(z)$.
\end{enumerate}
If a Maass form on $(\Gamma,\chi)$ vanishes in all the cusps of $M$,
it is called a \emph{Maass cusp form}.

\bigskip

The Laplace-Beltrami operator is an essentially self-adjoint operator on $M$,
hence its eigenvalues are real. The spectral resolution on $M$ consists of
three parts, the zero eigenvalue $\lambda_0=0$, spanned by the
constant eigenfunction, the continuous spectrum $\lambda\in[1/4,\infty)$,
spanned by certain Maass forms, the so called Eisenstein series,
and a countable set of discrete eigenvalues
$0<\lambda_1\leq\lambda_2\leq\ldots$ which are spanned by Maass cusp forms
\cite{Roelcke1966,Roelcke1967}.

We are particularly interested in the first discrete eigenvalue $\lambda_1$
which we compute numerically using Hejhal's algorithm
\cite{Hej99}\ignore{, see appendix~\ref{sec A},}
together with strategies for finding eigenvalues
\cite[end of \S2 and \S3]{The05} and \cite{The12}.

We compute on arithmetic reflection groups and on finite index subgroups
thereof.
If $\Gamma'$ is a finite index rotation or reflection subgroup of $\Gamma$,
then $\Gamma'$ is invariant with respect to conjugation by $\gamma\in\Gamma$.
The surface $M'=\Gamma'\backslash\h$ has symmetries and can be tesselated by
non-overlapping copies of the surface $M=\Gamma\backslash\h$,
\begin{align*}
M'=\cup_{\Gamma/\Gamma'}\gamma M.
\end{align*}

The Laplace-Beltrami operator is invariant with respect to isometries, and in
particular to the operations of the symmetry group $\Gamma/\Gamma'$.
Maass forms on $(\Gamma',\Id)$ fall into symmetry classes,
where the symmetry classes are represented by multiplicative characters $\chi$,
\begin{align*}
\chi(\gamma)f(\gamma z)=f(z) \quad \forall\gamma\in\Gamma.
\end{align*}
By automorphy on $(\Gamma',\Id)$ we have
\begin{align*}
f(\gamma z)=f(z) \quad \forall\gamma\in\Gamma'.
\end{align*}
Consequently, a Maass form on $(\Gamma',\Id)$ is a Maass form on
$(\Gamma,\chi)$, where $\chi$ is a multiplicative character on $\Gamma$
with $\chi(\gamma)=1\ \forall\gamma\in\Gamma'$.
The symmetry of the Maass form is specified by the values of $\chi$ on
$\Gamma/\Gamma'$. Since every symmetry operation $\gamma\in\Gamma/\Gamma'$
has finite order, the multiplicative character takes values that are roots
of unity.

\subsection{Example: $M_{19}$ and $N_{19}$}
\label{sec 4.1}

The index two subgroup $\Gamma'$ of $\Gamma=N(\Gamma_0(19))$ is generated by
\begin{align*} &
\gamma_1=\begin{pmatrix} 1 & 1 \\ 0 & 1 \end{pmatrix},
\gamma_3=\begin{pmatrix} \sqrt{19} & \frac{9}{\sqrt{19}} \\
2\sqrt{19} & \sqrt{19} \end{pmatrix},
\gamma_4=\begin{pmatrix} \sqrt{19} & \frac{6}{\sqrt{19}} \\
3\sqrt{19} & \sqrt{19} \end{pmatrix},
\gamma_7\gamma_5\gamma_7\gamma_5=\begin{pmatrix} 39 & 10 \\
152 & 39 \end{pmatrix},
\\ &
\gamma_7\gamma_5\gamma_7\gamma_1\gamma_5=\begin{pmatrix} 37 & 9 \\
152 & 37 \end{pmatrix},
\gamma_7\gamma_5\gamma_7\gamma_3\gamma_5=\begin{pmatrix} 7\sqrt{19} &
\frac{31}{\sqrt{19}} \\ 30\sqrt{19} & 7\sqrt{19} \end{pmatrix},
\gamma_7\gamma_5\gamma_7\gamma_4\gamma_5=\begin{pmatrix} 0 &
-\frac{1}{\sqrt{19}} \\ \sqrt{19} & 0 \end{pmatrix},
\end{align*}
and $\Gamma'$ is invariant with respect to conjugation by
the involution and the reflection,
\begin{align*} &
\gamma_5=\begin{pmatrix} -\sqrt{19} & -\frac{5}{\sqrt{19}} \\
4\sqrt{19} & \sqrt{19} \end{pmatrix}
\quad \text{and} \quad
\gamma_7=\begin{pmatrix} -1 & 0 \\ 0 & 1 \end{pmatrix}.
\end{align*}
The conjugation relations of the generators are
\begin{align*}
& \gamma_5^{-1}\gamma_1\gamma_5=
(\gamma_7\gamma_5\gamma_7\gamma_5)^{-1}
(\gamma_7\gamma_5\gamma_7\gamma_1\gamma_5),
& \gamma_7^{-1}\gamma_1\gamma_7=\gamma_1^{-1}, \\
& \gamma_5^{-1}\gamma_3\gamma_5=
(\gamma_7\gamma_5\gamma_7\gamma_5)^{-1}
(\gamma_7\gamma_5\gamma_7\gamma_3\gamma_5),
& \gamma_7^{-1}\gamma_3\gamma_7=\gamma_3^{-1}, \\
& \gamma_5^{-1}\gamma_4\gamma_5=
(\gamma_7\gamma_5\gamma_7\gamma_5)^{-1}
(\gamma_7\gamma_5\gamma_7\gamma_4\gamma_5),
& \gamma_7^{-1}\gamma_4\gamma_7=\gamma_4^{-1}, \\
& \gamma_5^{-1}(\gamma_7\gamma_5\gamma_7\ \ \ \gamma_5)\gamma_5=
(\gamma_7\gamma_5\gamma_7\gamma_5)^{-1}\ \ \,,
& \gamma_7^{-1}(\gamma_7\gamma_5\gamma_7\ \ \ \gamma_5)\gamma_7=
(\gamma_7\gamma_5\gamma_7\ \ \ \gamma_5)^{-1}, \\
& \gamma_5^{-1}(\gamma_7\gamma_5\gamma_7\gamma_1\gamma_5)\gamma_5=
(\gamma_7\gamma_5\gamma_7\gamma_5)^{-1}\gamma_1,
& \gamma_7^{-1}(\gamma_7\gamma_5\gamma_7\gamma_1\gamma_5)\gamma_7=
(\gamma_7\gamma_5\gamma_7\gamma_1\gamma_5)^{-1}, \\
& \gamma_5^{-1}(\gamma_7\gamma_5\gamma_7\gamma_3\gamma_5)\gamma_5=
(\gamma_7\gamma_5\gamma_7\gamma_5)^{-1}\gamma_3,
& \gamma_7^{-1}(\gamma_7\gamma_5\gamma_7\gamma_3\gamma_5)\gamma_7=
(\gamma_7\gamma_5\gamma_7\gamma_3\gamma_5)^{-1}, \\
& \gamma_5^{-1}(\gamma_7\gamma_5\gamma_7\gamma_4\gamma_5)\gamma_5=
(\gamma_7\gamma_5\gamma_7\gamma_5)^{-1}\gamma_4,
& \gamma_7^{-1}(\gamma_7\gamma_5\gamma_7\gamma_4\gamma_5)\gamma_7=
(\gamma_7\gamma_5\gamma_7\gamma_4\gamma_5)^{-1}.
\end{align*}

The surface $N_{19}=\Gamma'\backslash\h$ has two inequivalent cusps.
A fundamental domain for $N_{19}$ is displayed in Figure~\ref{fig:Dfd19}.
\begin{figure}
\includegraphics[width=.33\textwidth]{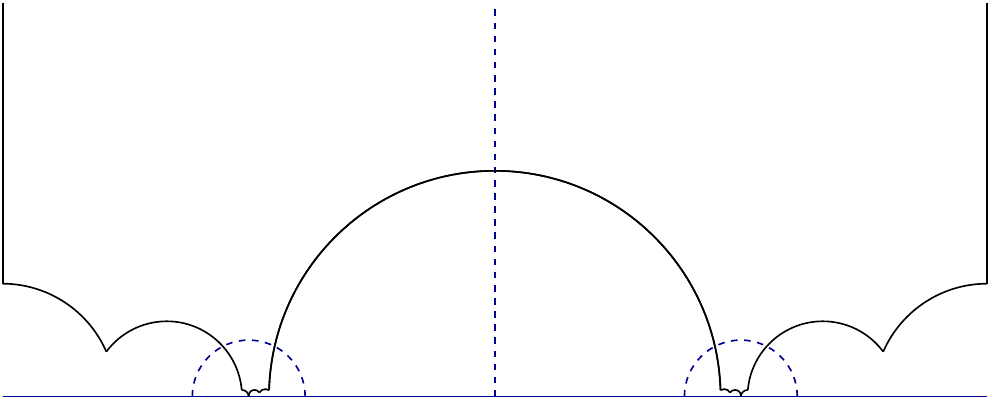}
\caption{\label{fig:Dfd19}A fundamental domain for $N_{19}$. $N_{19}$ is
symmetric with respect to the involution $\gamma_5$ and the reflection
$\gamma_7$. The symmetries are indicated as dotted lines.}
\end{figure}

Maass forms on $(\Gamma',\Id)$ fall into four symmetry classes,
\begin{itemize}
\item[\texttt{(++)}] $f(\gamma_7z)=f(z)$, $f(\gamma_5z)=f(z)$,
\item[\texttt{(+-)}] $f(\gamma_7z)=f(z)$, $f(\gamma_5z)=-f(z)$,
\item[\texttt{(-+)}] $f(\gamma_7z)=-f(z)$, $f(\gamma_5z)=f(z)$,
\item[\texttt{(--)}] $f(\gamma_7z)=-f(z)$, $f(\gamma_5z)=-f(z)$.
\end{itemize}
For ease of notation, we will refer to these four symmetry classes as
\texttt{(++)}, \texttt{(+-)}, \texttt{(-+)}, and \texttt{(--)}, respectively.

Let $\Gamma=N(\Gamma_0(19))$ be the orientation preserving
arithmetic maximal reflection group and let
$\hat{\Gamma}=\langle\Gamma,\gamma_7\rangle$ be
the non-orientation preserving arithmetic maximal reflection group. Moreover,
let $M_{19}=\Gamma\backslash\h$ and $\hat{M}_{19}=\hat{\Gamma}\backslash\h$
be the corresponding quotient surfaces.
Maass forms on $(\Gamma',\Id)$ with symmetry $\chi$ are Maass forms on
$(\hat{\Gamma},\chi)$ and vice-versa.

Computing Maass cusp forms on $(\hat{\Gamma},\chi)$, we need to take care of
one cusp only, while we obtain the discrete spectra on $N_{19}$ and $M_{19}$,
\begin{align*} &
\spec(N_{19})=
\spec(\hat{\Gamma},\texttt{(++)}) \cup \spec(\hat{\Gamma},\texttt{(+-)}) \cup
\spec(\hat{\Gamma},\texttt{(-+)}) \cup \spec(\hat{\Gamma},\texttt{(--)}), \\ &
\spec(M_{19})=
\spec(\hat{\Gamma},\texttt{(++)}) \cup \spec(\hat{\Gamma},\texttt{(-+)}).
\end{align*}
The first few eigenvalues of $\spec(\hat{\Gamma},\chi)$, and
hence, for $N_{19}$ and $M_{19}$ are listed in Table~\ref{tab:n=19}.

\begin{table}
\caption{\label{tab:n=19}Consecutive eigenvalues $\lambda$ of Maass cusp forms
on the index two subgroup $\Gamma'$ of $\Gamma=N(\Gamma_0(19))$ in dependence
of the symmetry class. We note that there is only one eigenvalue for symmetry
class~\texttt{(+-)}.}
\begin{tabular}{cccc}
$\lambda$ & $\lambda$ & $\lambda$ & $\lambda$ \\
\tiny{symmetry clsss~\texttt{(++)}} & \tiny{symmetry class~\texttt{(+-)}} &
\tiny{symmetry class~\texttt{(-+)}} & \tiny{symmetry class~\texttt{(--)}} \\
\hline
2.018365089 & 0.169612040 & 5.526238214 & 5.623644096 \\
5.178444802 & & 10.44255994 & 11.09994133 \\
8.183549674 & & 11.98729193 & 13.34726954 \\
9.893410725 & & 17.09642816 & 19.16508797 \\
16.96819424 & & 19.89417157 & 21.71327432 \\
18.55690731 & & 23.09408224 & 24.65978283 \\
18.59544749 & & 27.31732822 & 28.16059477 \\
21.32460096 & & 30.03293145 & 30.44847115 \\
24.00366338 & & 31.23572386 & 36.66774935 \\
30.96013885 & & 38.37766641 & 37.67237085 \\
32.92132204 & & 40.77536109 & 40.74040666 \\
33.39522368 & & 41.99782317 & 43.19669501 \\
36.12307204 & & 45.98970332 & 46.18071110 \\
41.54143328 & & 49.35948927 & 49.21126748 \\
41.84058775 & & 49.68599715 & 52.24479435 \\
44.25260852 & & 55.81902339 & 54.63308313 \\
46.79262854 & & 58.65413871 & 57.51914857 \\
51.16589334 & & 60.25006194 & 61.12735841 \\
54.93920130 & & 63.93528420 & 63.81506781 \\
\vdots & & \vdots & \vdots \\
\end{tabular}
\end{table}

\subsection{Maass cusp forms on arithmetic maximal reflection groups with two
ideal vertices}
\label{sec 4.3}

There are ten non-cocompact arithmetic maximal reflection groups with two
ideal vertices which we consider \cite[\S4.3]{LakelandThesis}. These are the groups of
level $11,15,17,19,22,26,33,42,55,66$ on which we compute Maass cusp forms
numerically.

Similarly as before, Maass forms on these groups fall into four symmetry
classes.
\ignore{%
\begin{itemize}
\item[\texttt{(++)}] $f(x+iy) =f(-x+iy)$, $f(g z) =f(z)$,
\item[\texttt{(+-)}] $f(x+iy) =f(-x+iy)$, $f(g z) =-f(z)$,
\item[\texttt{(-+)}] $f(x+iy) =-f(-x+iy)$, $f(g z) =f(z)$,
\item[\texttt{(--)}] $f(x+iy) =-f(-x+iy)$, $f(g z) =-f(z)$,
\end{itemize}
where $g$ is an elliptic element of order two not in the group, but is a
generator of the index two supergroup $N(\Gamma_0(n))$.

}%
For each of the symmetry classes~\texttt{(++)}, \texttt{(-+)}, and
\texttt{(--)}, there are infinitely many Maass cusp forms per level.
For the symmetry class~\texttt{(+-)} there are also infinitely many Maass
cusp forms for level $15$ and $17$, but there is at most one Maass cusp
form per level $11,19,22,26,33,42,55,66$. The reason is that for the
latter levels the group is not congruence.

We are interested in small eigenvalues. The number of Maass cusp forms
with eigenvalue less than $200$ is listed in Table~\ref{tab:2c} in
dependence of the level and the symmetry class. The smallest eigenvalue
in dependence of the level and the symmetry class is listed in
Table~\ref{tab:2}. The upper bounds of \cite[Table~4.2]{LakelandThesis}
are respected.

\begin{table}
\caption{\label{tab:2c}The number of Maass cusp forms with eigenvalue less
than $200$ in dependence of the level and the symmetry class.}
\begin{tabular}{c|cccc}
& $\#\{\lambda\in(0,200]\}$ & $\#\{\lambda\in(0,200]\}$ &
$\#\{\lambda\in(0,200]\}$ & $\#\{\lambda\in(0,200]\}$ \\
level & \tiny{symmetry class~\texttt{(++)}} &
\tiny{symmetry class~\texttt{(+-)}} &
\tiny{symmetry class~\texttt{(-+)}} &
\tiny{symmetry class~\texttt{(--)}} \\
\hline
11 & 37 & 1 & 38 & 38 \\
15 & 36 & 37 & 38 & 38 \\
17 & 63 & 58 & 63 & 60 \\
19 & 71 & 1 & 68 & 68 \\
22 & 60 & 1 & 61 & 61 \\
26 & 74 & 1 & 72 & 72 \\
33 & 86 & 0 & 84 & 84 \\
42 & 85 & 0 & 83 & 83 \\
55 & 140 & 1 & 128 & 127 \\
66 & 138 & 1 & 129 & 128 \\
\end{tabular}
\end{table}

\begin{table}
\caption{\label{tab:2}The smallest eigenvalue on non-cocompact arithmetic
maximal reflection groups in dependence of the level and the symmetry
class. On the non-congruence subgroups of level $33$ and $44$ there are
no Maass cusp forms of symmetry class~\texttt{(+-)}.}
\begin{tabular}{c|cccc}
& $\lambda$ & $\lambda$ & $\lambda$ & $\lambda$ \\
level & \tiny{symmetry class~\texttt{(++)}} &
\tiny{symmetry class~\texttt{(+-)}} & \tiny{symmetry class~\texttt{(-+)}} &
\tiny{symmetry class~\texttt{(--)}} \\
\hline
11 & 6.41822455110 & 0.24456267323 & 9.06024545639 &10.00253800339 \\
15 &10.62011651411 & 3.56777601683 & 5.82589930819 & 9.42106297240 \\
17 & 3.67134534972 & 0.25000000000 & 4.12297031172 & 6.24399453828 \\
19 & 2.01836508907 & 0.16961204041 & 5.52623821363 & 5.62364409573 \\
22 & 3.24185004329 & 0.23828308162 & 4.93732929346 & 5.81410275753 \\
26 & 2.14598046442 & 0.21588714472 & 4.46353001031 & 5.05501819375 \\
33 & 2.24813673858 & & 2.67134067803 & 4.83385807583 \\
42 & 1.87318457587 & & 2.92045953841 & 4.43972502521 \\
55 & 0.68768162820 & 0.14843098682 & 2.76461871050 & 3.30378314566 \\
66 & 0.62303205749 & 0.15286533999 & 2.67134067803 & 3.43082014307 \\
\end{tabular}
\end{table}

The first discrete eigenvalue on $N_{17}$ is elusive. It belongs to a CM form.

\subsection{CM forms}
\label{sec 4.4}

CM forms are Maass cusp forms that arise as theta-lifts from $\gl1$.
CM forms with eigenvalue $\lambda=1/4$ were first constructed by Hecke
\cite{He27}. Later, Maass gave a more general construction of CM forms
and identified them with Maass cusp forms \cite{Maa49}.

CM forms live on congruence subgroups of certain levels.
On arithmetic maximal reflection groups with two ideal vertices,
there are CM forms for level $n=17$ (only). By construction, they are
in symmetry class~\texttt{(+-)}.

Consider the level $n=17$. Let $F=\Q(\sqrt{17})$ be a real quadratic field.
Let $\eta=4+\sqrt{17}$ be the usual fundamental unit for $\mathcal{O}_F$.
Let $k$ be an arbitrary integer and $r=\frac{\pi k}{\log \eta}$
be the spectral parameter. Then, $\lambda=r^2+1/4$ is the eigenvalue of a
CM form on the surface $N_{17}$. If $k\ne0$, the Fourier expansion
coefficients of these CM forms are explicitely given in \cite{HS01},

The CM form for $k=0$ is special and needs to be treated with care.
According to the definition of Maass cusp forms, they should be
square-integrable. Obviously, this is the case for all Maass cusp forms
(CM forms included), except for the CM form with eigenvalue $\lambda=1/4$.
The latter has the term $a_0y^{1/2}$ in its Fourier expansion
whose $L^2$-norm diverges logarithmically.

The CM form with eigenvalue $\lambda=1/4$ is elusive.
Its $L^2$-norm exists in the sense of a distribution and its Fourier
coefficients read
\begin{align*}
a_0= \log{\eta}, \quad
a_1= 1, \quad
a_p= \left(\frac{n}{p}\right)_{\text{Kr}}+1 \quad \text{for $p$ prime}, \quad
a_{n_1}a_{n_2}=
\sum_{l \mid (n_1,n_2)} \left(\frac{n}{l}\right)_{\text{Kr}} a_{n_1 n_2/l^2},
\end{align*}
where $\left(\frac{n}{l}\right)_{\text{Kr}}$ is the Kronecker symbol, a
multiplicative character modulo the level $n$.

\section{Conclusions}
\label{sec 5}

The most immediate conclusion to be drawn from the results above is that there
exist maximal arithmetic hyperbolic reflection groups with $\lambda_1 < 1/4$.
Further, the first eigenvalue for these groups can be lower than it must be
for congruence groups, and so one cannot hope that there is a lower bound for
$\lambda_1$ corresponding with the known lower bounds for congruence groups.
The smallest first eigenvalue found here is 0.14843, but we have no reason to
believe that this should itself serve as a lower bound for $\lambda_1$ amongst
all maximal arithmetic hyperbolic reflection groups.

We note that for all of the examples considered here, the Cheeger constant
$h \leq 1$, which is true because these are cusped hyperbolic surfaces, and a
cusp neighborhood always has Cheeger ratio equal to 1. In all of the cases
here where $\lambda_1 \leq 1$ on $N_n$, we find that $\lambda_1 < h$, so the
Cheeger constant serves as an upper bound for the first eigenvalue. Further
investigation may be helpful here in order to understand whether, and in what
circumstances, we may take the Cheeger constant to serve as an upper bound for
$\lambda_1$, and in which circumstances we have $\lambda_1 < 1$.

It is interesting to compare each
Cheeger constant $h(M_n)$ with that of its double cover $h(N_n)$. Recall that
each surface $M_n$ is congruence. With the exception of the case $n=33$, when
the cover $N_n$ is not congruence, the ratio $\dfrac{h(N_n)}{h(M_n)}$ is less
than $1/2$. In both cases where $N_n$ is congruence, the ratio is larger than
$1/2$. This raises the following question.\\

\noindent {\bf Question.} Suppose $\Gamma$ is a congruence subgroup
commensurable with $\mathrm{PSL}_2(\mathbb Z)$, $\Gamma' < \Gamma$ an index
two subgroup, and that
$h(\mathbb{H}^2 / \Gamma') / h(\mathbb{H}^2 / \Gamma) < 1/2$. Must $\Gamma'$
then be a non-congruence group?\\

%
%

It is perhaps interesting to note that the cases $n=15$ and $n=42$ produce examples of orbifold surfaces
$N_{15}$ and $N_{42}$ where $\lambda_1^{\mbox{disc}}$ is so large that using it instead of the spectral gap (in this case $1/4$) would result in examples which appear to fail to satisfy Buser's inequality.
Specifically, for $N_{15}$, we have $h=0.498728$, and $\lambda_1^{\mbox{disc}} = 3.5678$,
which exceeds $2h+10h^2 = 3.4848$. For $N_{42}$ we have $h=0.328406$ and
$\lambda_1^{\mbox{disc}} = 1.87318$, which exceeds $2h+10h^2 = 1.7353$. Recall that Buser's inequality is formulated for closed manifold surfaces, and our examples have cusps and cone points. Nevertheless, these examples are worthy of further investigation, as this observation may indicate the presence of some different geometric properties which correspond to higher discrete eigenvalues.

\bibliographystyle{amsplain}
\bibliography{cheegerrefs}

\end{document}